\documentclass{article}
\usepackage{amssymb}
\usepackage{amsmath}
\usepackage{amsthm}
\usepackage[margin=1.75in]{geometry}
\usepackage{graphicx}

\newcommand{\rmd}{\mathrm{d}}

\newcommand{\rmC}{\mathrm{C}}
\newcommand{\rmI}{\mathrm{I}}

\newcommand{\bbR}{\mathbb{R}}

\newcommand{\abs}[1]{|{#1}|}

\newcommand{\bigparen}[1]{\bigl({#1}\bigr)}
\newcommand{\Bigparen}[1]{\Bigl({#1}\Bigr)}

\newcommand{\bigbracket}[1]{\bigl[{#1}\bigr]}

\newcommand{\set}[1]{\{{#1}\}}

\newcommand{\Bigset}[1]{\Bigl\{{#1}\Bigr\}}

\newcommand{\norm}[1]{\|{#1}\|}
\newcommand{\bignorm}[1]{\bigl\|{#1}\bigr\|}

\newcommand{\ip}[2]{\langle{#1},{#2}\rangle}
\newcommand{\bigip}[2]{\bigl\langle{#1},{#2}\bigr\rangle}

\newcommand{\biggip}[2]{\biggl\langle{#1},{#2}\biggr\rangle}

\newcommand{\argmin}[1]{\underset{#1}{\mathrm{argmin}}}

\newtheorem{theorem}{Theorem}

\title{Optimal frames and Newton's method}
\author{Matthew Fickus\thanks{Air Force Institute of Technology, Department of Mathematics, Wright-Patterson AFB, OH 45433, Matthew.Fickus@afit.edu} \and Dustin G.~Mixon\thanks{Princeton University, Program in Applied and Computational Mathematics, Princeton, NJ 08544}}
\date{}

\begin{document}
\maketitle

\begin{abstract}
Given a parametrized family of finite frames, we consider the optimization problem of finding the member of this family whose coefficient space most closely contains a given data vector.  This nonlinear least squares problem arises naturally in the context of a certain type of radar system.  We derive analytic expressions for the first and second partial derivatives of the objective function in question, permitting this optimization problem to be efficiently solved using Newton's method.  We also consider how sensitive the location of this minimizer is to noise in the data vector.  We further provide conditions under which one should expect the minimizer of this objective function to be unique.  We conclude by discussing a related variational-calculus-based approach for solving this frame optimization problem over an interval of time.
\end{abstract}

%%%%%%%%%%%%%%%%%%%%%%%%%%%%%%%%%%%%%%%%%%%%%%%%%%%%%%%%%%%%%%%%%%%%%%%%%%%%%%%%%%%%%%%%%%%%%%%%%%%%%%%%%%%%%%%%%%%%%%%%%%%%%%%
\section{Introduction}

In frame theory, the \textit{synthesis operator} of a finite sequence of vectors $\set{f_n}_{n=1}^{N}$ in $\mathbb{R}^M$ is the operator $F:\bbR^N\rightarrow\bbR^M$, \smash{$Fw:=\sum_{n=1}^{N}w(n)f_n$}.  That is, $F$ is an $M\times N$ matrix whose $n$th column is $f_n$.  Taking the transpose of the synthesis operator yields the \textit{analysis operator} $F^*:\bbR^M\rightarrow\bbR^N$ given by $(F^* v)(n)=\ip{v}{f_n}$.  The sequence $\set{f_n}_{n=1}^{N}$ is a \textit{frame} for $\bbR^M$ if there exist \textit{frame bounds} $0<A\leq B<\infty$ such that $A\norm{v}^2\leq\norm{F^* v}^2\leq B\norm{v}^2$ for all $v\in\bbR^M$.  In this finite-dimensional setting, we have that $\set{f_n}_{n=1}^{N}$ is a frame for $\bbR^M$ if and only if it spans $\bbR^M$, which is equivalent to having its \textit{frame operator} $FF^*$ be invertible.

Frame theory was born of the theory of linear least squares: given an overdetermined system $F^* v=w$, the goal is to reconstruct $v$ from $F$ and $w$ by minimizing $\norm{F^* v-w}^2$.  The minimizers $v$ are characterized as the solutions to the normal equations $FF^* v=Fw$, which have a unique solution of $v=(FF^*)^{-1}F w$ if and only if $\set{f_n}_{n=1}^{N}$ is a frame for $\bbR^M$.  The minimum value of $\norm{F^* v-w}^2$ is thus:
\begin{equation}
\label{equation.least squares minimum}
\min_{v\in\bbR^M}\norm{F^* v-w}^2
=\norm{F^*(FF^*)^{-1}F w-w}^2
=\bignorm{\bigbracket{\rmI-F^*(FF^*)^{-1}F}w}^2.
\end{equation}

In this paper, we focus on a generalization of these ideas that, as detailed below, arises naturally in a certain radar problem.  In this generalization, we are not given a single frame $\set{f_n}_{n=1}^{N}$ for $\bbR^M$ but rather a parametrized family of frames $\set{f_n(x)}_{n=1}^{N}$ for $\bbR^M$, where the parameter vector $x$ lies in some subset $\Omega$ of $\bbR^P$.  That is, we have an $M\times N$ synthesis matrix $F(x)$, each of whose entries depend on $P$ real parameters.  In particular, given $F(x)$ and $w$, our goal is to solve:
\begin{equation}
\label{equation.nonlinear least squares}
\argmin{x\in\Omega}\min_{v\in\mathbb{R}^M}\norm{F^*(x)v-w}^2.
\end{equation}
That is, we want to find the parameters $x$ for which the range of the corresponding analysis operator $F^*(x)$ is as close as possible to $w$.  In a radar application detailed below, the optimal $x$ and $v$ correspond to the unknown position and velocity of a target of interest, $F(x)$ has an explicit formula in terms of the locations of a set of radar transmitters and receivers, and $w$ is a collection of Doppler measurements; solving~\eqref{equation.nonlinear least squares} corresponds to determining the target's location based on Doppler information alone.  More generally, solving~\eqref{equation.nonlinear least squares} corresponds to finding the particular frame from the family of frames $\set{f_n(x)}_{n=1}^{N}$ which is most consistent with the measured data $w$. 

The first step in solving the nonlinear least squares problem~\eqref{equation.nonlinear least squares} is to recognize that for any fixed $x\in\Omega$, the minimization over $v$ is but a simple linear least squares problem.  That is, in light of \eqref{equation.least squares minimum}, solving~\eqref{equation.nonlinear least squares} reduces to minimizing the \textit{error function} $E:\Omega\rightarrow[0,\infty)$, 
\begin{equation}
\label{equation.definition of E}
E(x)
:=\bignorm{\set{\rmI-F^*(x)[F(x)F^*(x)]^{-1}F(x)}w}^2.
\end{equation}

The bulk of this paper is devoted to the study of the minimization of~\eqref{equation.definition of E}.  Our first priority is to develop a practical method for finding the optimal $x$.  However, since $F$ can depend nonlinearly on $x$, it is unrealistic to expect a nice closed-form solution for this optimal $x$ in terms of $F$ and $w$.  We thus settle for a good numerical algorithm to iteratively compute $x$, namely Newton's method.  This method relies on explicit expressions for the gradient and Hessian of~\eqref{equation.definition of E}, which are presented in the following section.  In Section~3, we provide some results on the uniqueness of the optimal $x$ as well as the sensitivity of \eqref{equation.definition of E} to noise in $w$.  In the final section, we discuss an alternative variational-calculus-based approach for solving a time-varying version of~\eqref{equation.nonlinear least squares} in which $x$ and $w$ are functions of time and $v$ is the time-derivative of $x$.  In the remainder of the introduction, we detail our motivating radar application and highlight the relevant literature.

\subsection{A radar application}
In the radar community, the process of determining a target's location is known as \textit{localization}.  To be precise, let $x:\bbR\rightarrow\bbR^M$ denote the trajectory of the target.  That is, $x(t)$ denotes the location of the target at any given time $t$, and the dimension $M$ is typically either $2$ or $3$.  The goal of a given radar system is to determine $x(t)$ by bouncing electromagnetic signals off of the target.

Here, we focus on \textit{static} radar: the target is mobile, but the hardware that transmits and receives our radar signals is not.  More specifically, we consider \textit{multistatic radar}, which makes use of multiple fixed pairs of transmitters and receivers.  Let $N$ denote the number of these pairs, and let $\set{a_n}_{n=1}^{N}$ and $\set{b_n}_{n=1}^{N}$ in $\bbR^M$ denote the locations of the transmitters and receivers, respectively.  A signal broadcast from $a_n$, bounced off of the target located at $x(t)$, and then received at $b_n$ travels a total distance of $\varphi_n(x(t))$, where $\varphi_n:\bbR^M\rightarrow\bbR$ is the $n$th \textit{bistatic distance} function:
\begin{equation}
\label{equation.bistatic distance}
\varphi_n(x):=\norm{x-a_n}+\norm{x-b_n}.
\end{equation}

In the most simple radar systems, each transmitter broadcasts a sequence of pulses.  In the radar literature, the time lag between the transmitted and received pulses is known as the \textit{time difference of arrival} (TDOA).  By multiplying the TDOA by the speed of light, one obtains a set of $N$ real-valued measurements $\set{y_n(t)}_{n=1}^{N}$ which serve as estimates of $\set{\varphi_n(x(t))}_{n=1}^{N}$.  The target's true position $x(t)$ is then known to lie on the intersection of prolate spheroids of the form $\set{x\in\bbR^M : \varphi_n(x)=y_n(t)}$.

We focus on a different class of radar systems in which each transmitter broadcasts a continuous wave of constant frequency.  The periodic nature of such signals makes it nearly impossible to accurately measure TDOA.  Rather, one instead measures the \textit{frequency difference of arrival} (FDOA), namely the change in frequency between the transmitted and received wave.  This FDOA is proportional to the time-derivative of the bistatic distance:
\begin{equation*}
\frac{\rmd}{\rmd t}\varphi_n(x(t))
=\ip{\dot{x}(t)}{\nabla \varphi_n(x(t))},
\end{equation*}
a fact known as the \textit{Doppler effect}.
 
In short, we focus on FDOA multistatic radar, in which the goal is to determine $x(t)$ from a number of FDOA measurements $\set{w_n(t)}_{n=1}^{N}$, where each $w_n(t)$ should equal $\nabla \varphi_n(x(t))\cdot\dot{x}(t)$, up to measurement error and noise.  Here, one's first instinct is to try to explicitly solve the following system of $N$ first-order ordinary differential equations:
\begin{equation}
\label{equation.system of differential equations}
\ip{\dot{x}(t)}{\nabla \varphi_n(x(t))}=w_n(t),\quad\forall n=1,\dotsc,N.
\end{equation}
The problem with this approach is that when $M<N$, the system \eqref{equation.system of differential equations} is overdetermined and so it is not likely to be solvable.  This may be remedied with a standard least-squares approach, namely minimizing:
\begin{equation}
\label{equation.least squares regularization}
\sum_{n=1}^{N}\abs{\ip{\dot{x}(t)}{\nabla \varphi_n(x(t))}-w_n(t)}^2.
\end{equation}
Indeed, in the final section of this paper, we apply a variational approach to determine the Euler-Lagrange equation that any minimizer $x$ of the time-integral of \eqref{equation.least squares regularization} must satisfy.  However, the bulk of this paper is devoted to a more simple approach, namely using frame theory to tackle the minimization of~\eqref{equation.least squares regularization}.  Here, the key idea is to minimize~\eqref{equation.least squares regularization} \textit{individually} at each time $t$.

To be precise, let us now assume that FDOA measurements $\set{w_n(t)}_{n=1}^{N}$ are only available at a single time $t_0$.  In this setting, the target's velocity at that instant is essentially an unknown variable which is independent from its position.  To simplify notation, we let $x:=x(t_0)$ and $v:=\dot{x}(t_0)$ in $\bbR^M$ denote this unknown position and velocity, and let $w\in\bbR^N$ denote the vector of FDOA measurements $\set{w_n(t_0)}_{n=1}^{N}$.  Denoting the vector $\nabla\varphi_n(x)$ as $f_n(x)$, the quantity~\eqref{equation.least squares regularization} at $t=t_0$ can then be rewritten as:
\begin{equation*}
\sum_{n=1}^{N}\abs{\ip{v}{f_n(x)}-w_n}^2
=\norm{F^*(x)v-w}^2.
\end{equation*}
Thus, our problem of minimizing~\eqref{equation.least squares regularization} for $t=t_0$ reduces to our frame optimization problem~\eqref{equation.nonlinear least squares} in the special case where $P=M$ and the $n$th frame element $f_n(x)$ is the gradient of the $n$th bistatic distance function~\eqref{equation.bistatic distance}; one may quickly show that this gradient is the sum of two unit vectors, namely the vectors pointing to the target from the $n$th transmitter and receiver, respectively:
\begin{equation}
\label{equation.multistatic frame}
f_n(x)
=\nabla\varphi_n(x)
=\frac{x-a_n}{\norm{x-a_n}}+\frac{x-b_n}{\norm{x-b_n}}.
\end{equation}
With perfect FDOA measurements, the target's true location corresponds to the global minimizer of the error function~\eqref{equation.definition of E}.  In the third section, we consider the localization error that results from imperfect measurements.  We also consider the uniqueness of this minimizer: the results suggest that in order to uniquely localize a target, one wants the number of FDOA measurements $N$ to be at least $2M$.  This makes intuitive sense since these measurements depend on $2M$ real-variable unknowns: $M$ unknowns for the target's position, and an additional $M$ unknowns for the target's velocity.

\subsection{Relevant literature}
Optimization is a significant tool in frame theory, see~\cite{BenedettoF:03,MasseyR:10} for example.  The first derivatives of certain frame-theoretic quantities are presented in~\cite{CasazzaF:09,CasazzaFM:11}; we build on these results here, computing both first and second derivatives of the more complicated quantity~\eqref{equation.definition of E}.

The motivating FDOA radar localization problem has long been a subject of interest in the radar community.  Early work focused on the radar-based tracking of transmitters moving in ballistic trajectories~\cite{WeinsteinL:80}, and a formal analysis of the underlying theory~\cite{Shensa:81}.  A least-squares formulation of this localization problem is given in~\cite{ChanJ:90,ChanT:91,ChanT:92}.  There, an approximate solution to~\eqref{equation.nonlinear least squares} was found by explicitly evaluating the error function~\eqref{equation.definition of E} over a large grid.  It has since been suggested that the location of aircraft may be determined solely by measuring the Doppler effect that their motion induces in the carrier waves of television broadcasts~\cite{PoullinL:94}.

%%%%%%%%%%%%%%%%%%%%%%%%%%%%%%%%%%%%%%%%%%%%%%%%%%%%%%%%%%%%%%%%%%%%%%%%%%%%%%%%%%%%%%%%%%%%%%%%%%%%%%%%%%%%%%%%%%%%%%%%%%%%%%%
\section{Newton's method}

Newton's method is a popular algorithm for solving nonlinear least squares problems.  We use it to minimize our error function $E(x)$ defined in~\eqref{equation.definition of E}.  Newton's method is an iterative algorithm and will converge to the minimizer at a quadratic rate provided $E(x)$ is sufficiently well-behaved and we make a good initial guess~\cite{NocedalW:99}.   A formal analysis of when $E(x)$ meets the necessary criteria for convergence is given in~\cite{Mixon:06}.  In practice, we make our initial guess $x_0$ by evaluating $E(x)$ over a grid, and choosing the grid point which yields the smallest value.

Given a current guess $x_k$, Newton's method approximates $E(x)$ as a paraboloid (second-order Taylor multinomial) on a neighborhood of $x_k$ and then moves in the direction of that paraboloid's vertex.  Explicitly, we let:
\begin{equation}
\label{equation.Newton's method}
x_{k+1}:=x_k-\gamma[(\nabla^2 E)(x_k)]^{-1}(\nabla E)(x_k),
\end{equation}
where $0<\gamma<1$ is some experimentally chosen step-size parameter and $(\nabla E)(x)$ and $(\nabla^2 E)(x)$ are the gradient and Hessian of~\eqref{equation.definition of E}, respectively.  We compute this gradient and Hessian using a type of differential calculus for matrix-valued functions.

To be precise, for an open subset $\Omega$ of $\bbR^P$, let $\rmC^1(\Omega,\bbR^{N_1\times N_2})$ denote the set of matrix-valued functions $A(x)$ from $\Omega$ into the set of all $N_1\times N_2$ matrices with the property that each of the $P$ partial derivatives of each of the $N_1N_2$ entries of $A(x)$ exist and is continuous on $\Omega$.  Here, the partial derivative \smash{$\frac{\partial A}{\partial x_p}(x)$} of $A(x)$ with respect to the $p$th variable $x_p$ is obtained by computing the $p$-partial derivative of each entry of $A(x)$ independently.  Equivalently, letting $\set{\delta_p}_{p=1}^{P}$ be the identity basis for $\bbR^P$, we have:
\begin{equation*}
\frac{\partial A}{\partial x_p}(x):=\lim_{t\rightarrow 0}\frac{1}{t}[A(x+t\delta_p)-A(x)].
\end{equation*}
One can easily show that this derivative is linear and has a matrix-product rule:
\begin{equation*}
\frac{\partial}{\partial x_p}[A_1(x)A_2(x)]
=\frac{\partial A_1}{\partial x_p}(x)A_2(x)+A_1(x)\frac{\partial A_2}{\partial x_p}(x),
\end{equation*}
for all $A_1\in\rmC^1(\Omega,\bbR^{N_1\times N_2})$, $A_2\in\rmC^1(\Omega,\bbR^{N_2\times N_3})$.  Writing inner products on $\bbR^N$ in terms of matrix products, this further yields the inner-product rule:
\begin{equation*}
\frac{\partial}{\partial x_p}\ip{w_1(x)}{w_2(x)}
=\biggip{\frac{\partial w_1}{\partial x_p}(x)}{w_2(x)}+\biggip{w_1(x)}{\frac{\partial w_2}{\partial x_p}(x)},
\end{equation*}
for all $w_1,w_2\in\rmC^1(\Omega,\bbR^N)$.  We shall also make use of a type of quotient rule.  To be precise, if $A\in\rmC^1(\Omega,\mathbb{R}^{N\times N})$, then its entries, and therefore determinant, are continuous in $x$.  In particular, if $A(x_0)$ is invertible for some $x_0\in\Omega$, then $A(x)$ is also invertible on a neighborhood of $x_0$.  Moreover, the cofactor form of $A^{-1}(x)$ implies that it is itself continuous in $x$.  As such,
\begin{align*}
-A^{-1}(x)\frac{\partial A}{\partial x_p}(x)A^{-1}(x)
&=-\lim_{t\rightarrow0}A^{-1}(x+t\delta_p)\lim_{t\rightarrow 0}\frac1t[A(x+t\delta_p)-A(x)]A^{-1}(x)\\
&=\lim_{t\rightarrow0}\frac1t[A^{-1}(x+t\delta_p)-A^{-1}(x)].
\end{align*}
In particular, if $A(x)$ is invertible for all $x\in\Omega$, then $A^{-1}\in\rmC^1(\Omega,\mathbb{R}^{N\times N})$ with:
\begin{equation}
\label{equation.derivative of inverse}
\frac{\partial A^{-1}}{\partial x_p}(x)
=-A^{-1}(x)\frac{\partial A}{\partial x_p}(x)A^{-1}(x).
\end{equation}
These facts in hand, we are ready to compute the partial derivatives of~\eqref{equation.definition of E}.  Let:
\begin{equation*}
\Pi(x):=\rmI-F^*(x)(F(x)F^*(x))^{-1}F(x).    
\end{equation*}
It is well known that $\Pi(x)$ is the orthogonal projection operator onto the null space of $F(x)$.  As such, our error function~\eqref{equation.definition of E} can be simplified as:
\begin{equation*}
E(x)
=\norm{\Pi(x)w}^2
=\ip{\Pi(x)w}{\Pi(x)w}
=\ip{w}{\Pi^*(x)\Pi(x)w}
=\ip{w}{\Pi(x)w}.
\end{equation*}
By the product rule, the $p$th derivative of $E(x)$ is thus $\frac{\partial E}{\partial x_p}(x)=\ip{w}{\frac{\partial\Pi(x)}{\partial x_p}w}$.  To compute this derivative of $\Pi$, we first use~\eqref{equation.derivative of inverse} and the product rule to find the corresponding partial derivative of $[F(x)F^*(x)]^{-1}$; here, for the sake of succinctness and readability, we shorten ``$F(x)$" to simply ``$F$":
\begin{equation}
\label{equation.gradient computation 1}
\frac{\partial}{\partial x_p}(FF^*)^{-1}=-(FF^*)^{-1}\Bigparen{\frac{\partial F}{\partial x_p}F^*+F\frac{\partial F^*}{\partial x_p}}(FF^*)^{-1}.
\end{equation}
The product rule then gives:
\begin{align*}
\frac{\partial \Pi}{\partial x_p}
&=\frac{\partial}{\partial x_p}\bigbracket{\rmI-F^*(FF^*)^{-1}F}\\
&=-\frac{\partial F^*}{\partial x_p}(FF^*)^{-1}F-F^*(FF^*)^{-1}\frac{\partial F}{\partial x_p}\\
&\qquad+F^*(FF^*)^{-1}\Bigparen{\frac{\partial F}{\partial x_p}F^*+F\frac{\partial F^*}{\partial x_p}}(FF^*)^{-1}F.
\end{align*}
Combining the first and fourth terms above, as well as the second and third, gives:
\begin{align}
\nonumber
\frac{\partial \Pi}{\partial x_p}
&=-\bigbracket{\rmI-F^*(FF^*)^{-1}F}\frac{\partial F^*}{\partial x_p}(FF^*)^{-1}F-F^*(FF^*)^{-1}\frac{\partial F}{\partial x_p}\bigbracket{\rmI-F^*(FF^*)^{-1}F}\\
\label{equation.gradient computation 2}
&=-\Pi\,\Pi_p^*-\Pi_p^{}\Pi,
\end{align}
where we adopt the notation $\Pi_p:=F^*(FF^*)^{-1}\frac{\partial F}{\partial x_p}$.  Since $\Pi^*=\Pi$ and we are working with real-valued inner products, the $p$th  partial derivative of $E$ is thus:
\begin{equation}
\label{equation.gradient computation 3}
\frac{\partial E}{\partial x_p}
=\biggip{w}{\frac{\partial\Pi}{\partial x_p}w}
=\ip{w}{-\Pi\,\Pi_p^*-\Pi_p^{}\Pi w}
=-2\ip{w}{\Pi_p^{}\Pi w}.
\end{equation}
In light of~\eqref{equation.gradient computation 3}, computing second derivatives of $E$ involves computing the first derivatives of $\Pi_p$.  This computation parallels that of~\eqref{equation.gradient computation 2}, making renewed use of~\eqref{equation.gradient computation 1}.  To be precise, for any $p,q=1,\dotsc,P$,
\begin{align*}
\frac{\partial \Pi_p}{\partial x_q}
&=\frac{\partial}{\partial x_q}F^*(FF^*)^{-1}\frac{\partial F}{\partial x_p}\\
&=\frac{\partial F^*}{\partial x_q}(FF^*)^{-1}\frac{\partial F}{\partial x_p}+F^*(FF^*)^{-1}\frac{\partial^2 F}{\partial x_q\partial x_p}\\
&\qquad-F^*(FF^*)^{-1}\Bigparen{\frac{\partial F}{\partial x_q}F^*+F\frac{\partial F^*}{\partial x_q}}(FF^*)^{-1}\frac{\partial F}{\partial x_p}.
\end{align*}
Combining the first and fourth terms above yields:
\begin{equation*}
\frac{\partial \Pi_p}{\partial x_q}
=\Pi\frac{\partial F^*}{\partial x_q}(FF^*)^{-1}\frac{\partial F}{\partial x_p}+\Pi_{q,p}^{}-\Pi_q^{}\Pi_p^{},
\end{equation*}
where $\Pi_{q,p}:=F^*(FF^*)^{-1}\frac{\partial^2 F}{\partial x_q\partial x_p}$.  Putting $(FF^*)^{-1}(FF^*)$ in the first term gives:
\begin{align}
\nonumber
\frac{\partial \Pi_p}{\partial x_q}
&=\Pi\frac{\partial F^*}{\partial x_q}(FF^*)^{-1}FF^*(FF^*)^{-1}\frac{\partial F}{\partial x_p}+\Pi_{q,p}^{}-\Pi_q^{}\Pi_p^{}\\
\label{equation.gradient computation 4}
&=\Pi\,\Pi_q^{*}\Pi_p^{}+\Pi_{q,p}^{}-\Pi_q^{}\Pi_p^{}.
\end{align}
Having~\eqref{equation.gradient computation 2} and~\eqref{equation.gradient computation 4}, we take the $q$th partial derivative of~\eqref{equation.gradient computation 3}:
\begin{align*}
-\frac12\frac{\partial^2 E}{\partial x_q\partial x_p}
&=\biggip{w}{\frac{\partial \Pi_p}{\partial x_q}\Pi w}+\biggip{w}{\Pi_p\frac{\partial \Pi}{\partial x_q}w}\\
&=\bigip{w}{(\Pi\,\Pi_q^{*}\Pi_p^{}+\Pi_{q,p}^{}-\Pi_q^{}\Pi_p^{})\Pi w}+\bigip{w}{\Pi_p^{}(-\Pi\,\Pi_q^*-\Pi_q^{}\Pi)w}\\
&=\bigip{w}{(\Pi_{q,p}^{}\Pi+\Pi\,\Pi_q^{*}\Pi_p^{}\Pi-\Pi_q^{}\Pi_p^{}\Pi-\Pi_{p}^{}\Pi_q^{}\Pi-\Pi_p^{}\Pi\,\Pi_q^*)w}.
\end{align*}
We now rearrange this statement to better indicate the symmetry between $p$ and $q$ in this expression, which is consistent with the symmetry of mixed partial derivatives:
\begin{align}
\nonumber
\frac{\partial^2 E}{\partial x_q\partial x_p}
&=2\bigip{w}{(\Pi_{p}^{}\Pi_q^{}+\Pi_q^{}\Pi_p^{})\Pi w}+2\bigip{\Pi\,\Pi_p^*w}{\Pi\,\Pi_q^*w}\\
\label{equation.gradient computation 5}
&\qquad-2\bigip{\Pi_q^{}\Pi w}{\Pi_p^{}\Pi w}-2\bigip{w}{\Pi_{q,p}^{}\Pi w}.
\end{align}
We summarize~\eqref{equation.gradient computation 3} and~\eqref{equation.gradient computation 5} as the following result:
\begin{theorem}
\label{theorem.gradient and Hessian}
Let $\Omega$ be an open subset of $\bbR^P$ and let $F\in\rmC^2(\Omega,\mathbb{R}^{M\times N})$ where the columns $\set{f_n(x)}_{n=1}^{N}$ of $F(x)$ always form a frame for $\bbR^M$.  For any $w\in\mathbb{R}^N$, the first and second partial derivatives of~\eqref{equation.definition of E} are:
\begin{align*}
\frac{\partial E}{\partial x_p}
&=-2\ip{w}{\Pi_p^{}\Pi w},\\
\frac{\partial^2 E}{\partial x_q\partial x_p}
&=2\bigip{w}{(\Pi_{p}^{}\Pi_q^{}+\Pi_q^{}\Pi_p^{})\Pi w}+2\bigip{\Pi\,\Pi_p^*w}{\Pi\,\Pi_q^*w}\\
&\qquad-2\bigip{\Pi_q^{}\Pi w}{\Pi_p^{}\Pi w}-2\bigip{w}{\Pi_{q,p}^{}\Pi w},
\end{align*}
for all $p,q=1,\ldots,P$, where:
\begin{equation*}
\Pi:=\rmI-F^*(FF^*)^{-1}F,
\quad
\Pi_p:=F^*(FF^*)^{-1}\frac{\partial F}{\partial x_p},
\quad
\Pi_{q,p}:=F^*(FF^*)^{-1}\frac{\partial^2 F}{\partial x_q\partial x_p}.
\end{equation*}
\end{theorem}

Though the expressions for the gradient and Hessian of $E(x)$ in Theorem~\ref{theorem.gradient and Hessian} are complicated, they are nevertheless straightforward to implement in the Newton's method iteration~\eqref{equation.Newton's method}.  To be precise, given a current guess $x_k$, our first task is to compute \smash{$\frac{\partial F}{\partial x_p}(x_k)$} and \smash{$\frac{\partial^2 F}{\partial x_q\partial x_p}(x_k)$} for all $p,q=1,\dotsc,P$.  Calculating these derivatives columnwise, this is equivalent to finding \smash{$\frac{\partial f_n}{\partial x_p}(x_k)$} and \smash{$\frac{\partial^2 f_n}{\partial x_q\partial x_p}(x_k)$} for all $n=1,\dotsc,N$ and $p,q=1,\dotsc,P$.  For the particular frame~\eqref{equation.multistatic frame} that arises in FDOA multistatic radar, the first derivatives can be found by substituting $x_k-a_n$ and $x_k-b_n$ into the following easily-derived formula:
\begin{equation*}
\frac{\partial}{\partial x_p}\frac{x}{\norm{x}}
=\frac{1}{\norm{x}}\pi(x)\delta_p,
\end{equation*}
and summing the results; here, \smash{$\pi(x):=\rmI-\frac{xx^*}{\norm{x}^2}$} is the projection operator onto the orthogonal complement of the line passing through $x$, where $x^*$ denotes the transpose of the column vector $x$.  The second derivatives of $f_n(x)$ at $x_k$ can similarly be found using the relation:
\begin{equation*}
\frac{\partial^2}{\partial x_q\partial x_p}\frac{x}{\norm{x}}
=-\frac{1}{\norm{x}^3}\bigbracket{\pi(x)(\delta_p^{}\delta_q^*+\delta_q^{}\delta_p^*)+(\delta_p^*\pi(x)\delta_q^{})\rmI}x.
\end{equation*}

With \smash{$\frac{\partial F}{\partial x_p}(x_k)$} and \smash{$\frac{\partial^2 F}{\partial x_q\partial x_p}(x_k)$} in hand for all $p,q=1,\dotsc,P$, our second task in any given iteration~\eqref{equation.Newton's method} of Newton's method is to compute $(\nabla E)(x_k)$ and $(\nabla^2 E)(x_k)$ using Theorem~\ref{theorem.gradient and Hessian}; we now briefly outline an efficient means for doing so.  We begin by computing the synthesis operator $\tilde{F}(x_k):=[F(x_k)F^*(x_k)]^{-1}F(x_k)$ of the canonical dual frame $\set{\tilde{f}_n(x_k)}_{n=1}^{N}$.  Here, we emphasize that it is not necessary to explicitly compute the inverse of the frame operator $F(x_k)F^*(x_k)$.  Rather, the best algorithms for numerically computing $\tilde{F}(x_k)$, such as \textsc{Matlab}'s ``pinv" command, rely on methods of numerical linear algebra, such as QR factorization.  We then write the relevant operators in terms of $\tilde{F}^*(x_k)$:
\begin{align*}
\Pi(x_k)&=\rmI-\tilde{F}^*(x_k)F(x_k),\\
\Pi_p(x_k)&=\tilde{F}^*(x_k)\tfrac{\partial F}{\partial x_p}(x_k),\\
\Pi_{q,p}(x_k)&=\tilde{F}^*(x_k)\tfrac{\partial^2 F}{\partial x_q\partial x_p}(x_k),
\end{align*}
and compute, in the following order, the quantities:
\begin{gather}
\nonumber
\Pi(x_k)w,\quad \set{\Pi_p(x_k)\Pi(x_k) w}_{p=1}^{P},\quad \set{\Pi_p^*(x_k) w}_{p=1}^P,\\
\label{equation.computation quantities}
\set{\Pi(x_k)\Pi_p^*(x_k)w}_{p=1}^{P},\qquad \set{\Pi_{q,p}(x_k)\Pi(x_k)w}_{p,q=1}^P.
\end{gather}
Here, to be efficient, we make use of previous computations and exploit associativity to avoid costly matrix-matrix multiplications.  For example, to compute $\Pi_p(x_k)\Pi(x_k) w$, we take the previously computed $N\times 1$ vector $\Pi(x_k) w$, multiply it by the $M\times N$ matrix \smash{$\tfrac{\partial F}{\partial x_p}(x_k)$} and then multiply the resulting $M\times 1$ vector by the previously computed $N\times M$ matrix $\tilde{F}^*(x_k)$; this avoids the $\mathrm{O}(M^2N)$ cost of computing \smash{$\tilde{F}^*(x_k)\tfrac{\partial F}{\partial x_p}(x_k)$} directly.  

By rearranging some of the operators in Theorem~\ref{theorem.gradient and Hessian}, we see that every entry of $(\nabla E)(x_k)$ and $(\nabla^2 E)(x_k)$ can be found by computing inner products of the quantities \eqref{equation.computation quantities}.  Once this gradient and Hessian are found, we then compute $x_{k+1}$ according to~\eqref{equation.Newton's method}.  By iterating this process, we produce a sequence $\set{x_k}_{k=0}^{\infty}$ which hopefully converges to the minimizer of $E(x)$.  For the FDOA multistatic radar problem in particular, this approach seems to work well, often successfully localizing the target; see~\cite{Mixon:06} for extensive experimentation on simulated data.

%%%%%%%%%%%%%%%%%%%%%%%%%%%%%%%%%%%%%%%%%%%%%%%%%%%%%%%%%%%%%%%%%%%%%%%%%%%%%%%%%%%%%%%%%%%%%%%%%%%%%%%%%%%%%%%%%%%%%%%%%%%%%%%
\section{The sensitivity and uniqueness of minimizers}

For any given parameters $x\in\Omega\subseteq\bbR^P$, the quantity $E(x)$, as defined in~\eqref{equation.definition of E}, is the squared-distance of a given $w\in\bbR^N$ from the range of the analysis operator of the frame $\set{f_n(x)}_{n=1}^{N}$.  In the previous section, we discussed a numerical method for minimizing $E(x)$, that is, for finding the particular frame(s) which are most likely to have generated a given $w$.  In this section, we consider the uniqueness of such a minimizer $x$, as well as how sensitive it is to changes in $w$.

These two issues---sensitivity and uniqueness---are very important in real-world applications of this minimization problem.  For example, let us recall FDOA multistatic radar where $P=M$, $f_n(x)$ is given by~\eqref{equation.multistatic frame}, and $w$ is a list of Doppler-effect measurements, one for each of $N$ distinct pairs of transmitters and receivers.  Let $x_0$ and $v_0$ in $\bbR^M$ denote a target's position and velocity at a given instant, respectively.  By \eqref{equation.system of differential equations}, the $n$th component of $F^*(x_0)v_0\in\bbR^N$ is the instantaneous rate of change of the bistatic distance~\eqref{equation.bistatic distance}.  In a perfect radar system, the FDOA measurements $w$ would equal $F^*(x_0)v_0$.  However, due to various real-world issues, such as noise, quantization and an oversimplified physics model, our actual FDOA measurements are $w=F^*(x_0)v_0+\varepsilon$, where $\varepsilon\in\bbR^N$ is some hopefully small error vector.  The value of $E(x)$ at the target's true location $x_0$ is thus:
\begin{align}
\nonumber
E(x_0)
&=\bignorm{\set{\rmI-F^*(x_0)[F(x_0)F^*(x_0)]^{-1}F(x_0)}(F^*(x_0)v_0+\varepsilon)}^2\\
\label{equation.sensitivity analysis 1}
&=\bignorm{\set{\rmI-F^*(x_0)[F(x_0)F^*(x_0)]^{-1}F(x_0)}\varepsilon}^2.
\end{align}
In particular, if we make the unrealistic assumption that $\varepsilon=0$, then the target's true location $x_0$ is a global minimizer of $E(x)$, having value zero.  This begs the question: for $\varepsilon\neq0$, how far away is the minimizer of:
\begin{equation}
\label{equation.sensitivity analysis 2}
E(x):=\bignorm{\set{\rmI-F^*(x)[F(x)F^*(x)]^{-1}F(x)}(F^*(x_0)v_0+\varepsilon)}^2
\end{equation}
from $x_0$?  Though a complete answer to this question eludes us, we are nevertheless able to make two meaningful points.

First, since $\rmI-F^*(x_0)[F(x_0)F^*(x_0)]^{-1}F(x_0)$ is the projection operator onto the null space of $F(x_0)$ we have $E(x_0)\leq\norm{\varepsilon}^2$, with equality precisely when $\varepsilon$ lies in this null space.  Thus, the target must lie somewhere in the level set $\set{x\in\bbR^P: E(x)\leq\norm{\varepsilon}^2}$, the size of which is based both on the size of our measurement error $\varepsilon$ and the geometry of the surface $E(x)$.  Indeed, if $E(x)$ has small curvature at its minimizer, then even a slight error in $w$ may result in a large error in $x_0$.  The expressions for the gradient and Hessian of this surface, given in Theorem~\ref{theorem.gradient and Hessian}, are our first steps towards a better understanding of this geometry.

Second, we note that the minimizer of~\eqref{equation.sensitivity analysis 2} need not be unique even when $\varepsilon=0$.  For example if $N=M$, then for any $x\in\Omega$ the frame $\set{f_n(x)}_{n=1}^{N}$ is actually a basis for $\bbR^M$; this implies $F(x)$ is invertible and so $F^*(x)[F(x)F^*(x)]^{-1}F(x)=\rmI$.  In this case, we therefore have that \eqref{equation.sensitivity analysis 2} is identically zero, meaning every $x$ is a minimizer.  This is not good: for FDOA multistatic radar, this means that for any $x$, there exists a vector $v$ such that a target with that position and velocity would yield the measured Doppler vector $w=F^*(x_0)v_0+\varepsilon$; it is therefore impossible to localize the target.  

In practice, we address this uniqueness problem by adding more measurements.  Indeed, even when $\varepsilon=0$, the $N$-dimensional measurement vector $w=F^*(x_0)v_0$ depends on $M+P$ unknowns---the $P$-dimensional vector $x_0$ and the $M$-dimensional vector $v_0$---and so it's reasonable to believe that we need at least $N\geq M+P$ in order to guarantee that~\eqref{equation.sensitivity analysis 2} has $x_0$ as its unique minimizer.  To be precise, note that in this $\varepsilon=0$ case, the set of all minimizers of~\eqref{equation.sensitivity analysis 2} is equal to its set of zeros, namely:
\begin{equation}
\label{equation.definition of zero set}
\Bigset{x\in\Omega: \bigbracket{\rmI-F^*(x)[F(x)F^*(x)]^{-1}F(x)}w=0},
\end{equation}
which contains $x_0$.  And though \eqref{equation.definition of zero set} equals $\Omega$ for $N=M$, our numerical experiments~\cite{Mixon:06} indicate that making $Q$ additional measurements, that is, increasing $N$ to $M+Q$ for some $Q=1,\dotsc,P$, shrinks \eqref{equation.definition of zero set} down to a $(P-Q)$-dimensional submanifold of $\Omega$.  In particular, for $N\geq M+P$, the set of minimizers~\eqref{equation.definition of zero set} always seems to be discrete.  Though we are unable to formally prove that such behavior always holds, we are able to give the following partial result, which shows that when a certain set of $N$ vectors spans $\bbR^{M+P}$, then it is impossible for $E(x)$ to have a smooth continuum of minimizers.

\begin{theorem}
\label{theorem.uniqueness}
Let $\Omega$ be an open subset of $\bbR^P$ and let $F\in\rmC^1(\Omega,\mathbb{R}^{M\times N})$, where the columns $\set{f_n(x)}_{n=1}^{N}$ of $F(x)$ always form a frame for $\bbR^M$.  Given some $w\in\bbR^N$, if the vectors:
\begin{equation}
\label{theorem.uniqueness 1}
\set{f_n(x)\oplus Df_n(x)^*[F(x)F^*(x)]^{-1}F(x)w}_{n=1}^{N}
\end{equation}
form a frame for $\bbR^M\oplus\bbR^P$ for every $x\in\Omega$, then the zero set~\eqref{equation.definition of zero set} of the error function~\eqref{equation.definition of E} cannot contain a nonconstant smooth curve.  Here, $Df_n(x)$ denotes the $M\times P$ Jacobian matrix of $f_n$ at $x$.
\end{theorem}

\begin{proof}
We prove by contrapositive, assuming the set of zeros of $E(x)$ contains a nonconstant smooth curve and proving~\eqref{theorem.uniqueness 1} is not always a frame for $\bbR^M\oplus\bbR^P$.  To be precise, let $x\in\rmC^1((-\delta,\delta),\bbR^P)$ be a smooth curve of zeros of $E(x)$ and let $\hat{x}:=x(0)$, $\hat{u}:=\dot{x}(0)\neq 0$.

Adopting the shorthand $\Pi(x):=\rmI-F^*(x)[F(x)F^*(x)]^{-1}F^*(x)$ of the previous section, we have $0=E(x(t))=\norm{\Pi(x(t))w}^2$ for all $t\in(-\delta,\delta)$ and so $\Pi(x(t))w=0$ for all such $t$.  For any $n=1,\dotsc,N$, letting $\psi_n:\Omega\rightarrow\bbR$, $\psi_n(x)=\ip{\Pi(x)w}{\delta_n}$, we therefore have that $\psi_n(x(t))=0$ for all $t\in(-\delta,\delta)$.  By the chain rule, the derivative of this equation at $t=0$ is thus:
\begin{equation}
\label{equation.proof of uniqueness 1}
0
=\frac{\rmd}{\rmd t}\psi_n(x(t))\Bigr|_{t=0}
=\sum_{p=1}^{P}\frac{\partial\psi_n}{\partial x_p}(x(0))\frac{\rmd x_p}{\rmd t}(0)
=\sum_{p=1}^{P}\frac{\partial\psi_n}{\partial x_p}(\hat{x})\hat{u}(p),
\end{equation}
where $\hat{u}(p)$ denotes the $p$th entry of $\hat{u}\in\bbR^P$.  Now, the product rule gives:
\begin{equation*}
\frac{\partial\psi_n}{\partial x_p}(\hat{x})
=\frac{\partial}{\partial x_p}\ip{\Pi(x)w}{\delta_n}\Bigr|_{x=\hat{x}}
=\biggip{\frac{\partial\Pi}{\partial x_p}(x)w}{\delta_n}\Bigr|_{x=\hat{x}}.
\end{equation*}
At this point, a computation~\eqref{equation.gradient computation 2} from the previous section gives:
\begin{equation*}
\frac{\partial\psi_n}{\partial x_p}(\hat{x})
=-\bigip{[\Pi(\hat{x})\Pi_p^*(\hat{x})+\Pi_p^{}(\hat{x})\Pi(\hat{x})]w}{\delta_n}.
\end{equation*}
Moreover, since $\hat{x}$ is a zero of $E(x)$ we have $\Pi(\hat{x})w=0$, and so this simplifies to:
\begin{equation}
\label{equation.proof of uniqueness 2}
\frac{\partial\psi_n}{\partial x_p}(\hat{x})
=-\ip{\Pi(\hat{x})\Pi_p^*(\hat{x})w}{\delta_n}.
\end{equation}
Substituting~\eqref{equation.proof of uniqueness 2} into~\eqref{equation.proof of uniqueness 1} gives:
\begin{equation*}
0
=-\sum_{p=1}^{P}\ip{\Pi(\hat{x})\Pi_p^*(\hat{x})w}{\delta_n}\hat{u}(p)
=-\biggip{\Pi(\hat{x})\sum_{p=1}^{P}\hat{u}(p)\Pi_p^*(\hat{x})w}{\delta_n}.
\end{equation*}
Since $n$ is arbitrary, we have that \smash{$\Pi(\hat{x})\sum_{p=1}^{P}\hat{u}(p)\Pi_p^*(\hat{x})=0$}.  This implies that \smash{$\sum_{p=1}^{P}\hat{u}(p)\Pi_p^*(\hat{x})$} lies in the null space of $\Pi(\hat{x})=\rmI-F^*(\hat{x})[F(\hat{x})F(\hat{x})^*]^{-1}F(\hat{x})$, which is known to equal the range (column space) of the analysis operator $F^*(\hat{x})$.  In particular, there necessarily exists $\hat{v}\in\bbR^M$ such that:
\begin{equation*}
F^*(\hat{x})\hat{v}=\sum_{p=1}^{P}\hat{u}(p)\Pi_p^*(\hat{x})w.
\end{equation*}
As such, for any $n=1,\dots,N$,
\begin{equation}
\label{equation.proof of uniqueness 3}
0
=\biggip{F^*(\hat{x})\hat{v}-\sum_{p=1}^{P}\hat{u}(p)\Pi_p^*(\hat{x})w}{\delta_n}
=\ip{\hat{v}}{F(\hat{x})\delta_n}-\sum_{p=1}^{P}\hat{u}(p)\ip{w}{\Pi_p(\hat{x})\delta_n}.
\end{equation}
To simplify this statement, note that $F(\hat{x})\delta_n=f_n(\hat{x})$ and so:
\begin{align*}
\ip{w}{\Pi_p(\hat{x})\delta_n}
&=\biggip{w}{F^*(\hat{x})[F(\hat{x})F^*(\hat{x})]^{-1}\frac{\partial F}{\partial x_p}(\hat{x})\delta_n}\\
&=\biggip{[F(\hat{x})F^*(\hat{x})]^{-1}F(\hat{x})w}{\frac{\partial f_n}{\partial x_p}(\hat{x})}\\
&=\biggip{[F(\hat{x})F^*(\hat{x})]^{-1}F(\hat{x})w}{Df_n(\hat{x})\delta_p}\\
&=\bigparen{Df_n(\hat{x})^*[F(\hat{x})F^*(\hat{x})]^{-1}F(\hat{x})w}(p).
\end{align*}
As such,~\eqref{equation.proof of uniqueness 3} becomes:
\begin{align*}
0
&=\ip{\hat{v}}{f_n(\hat{x})}_{\bbR^M}-\ip{\hat{u}}{Df_n(\hat{x})^*[F(\hat{x})F^*(\hat{x})]^{-1}F(\hat{x})w}_{\bbR^P}\\
&=\bigip{\hat{v}\oplus(-\hat{u})}{f_n(\hat{x})\oplus Df_n(\hat{x})^*[F(\hat{x})F^*(\hat{x})]^{-1}F(\hat{x})w}_{\bbR^M\oplus\bbR^P}.
\end{align*}
In particular, since $\hat{u}\neq0$ then $\hat{v}\oplus(-\hat{u})$ is a nonzero vector which is orthogonal to:
\begin{equation*}
f_n(\hat{x})\oplus Df_n(\hat{x})^*[F(\hat{x})F^*(\hat{x})]^{-1}F(\hat{x})w
\end{equation*}
for all $n=1,\dotsc,N$, meaning this set of vectors is not a frame for $\bbR^M\oplus\bbR^P$.
\end{proof}

%%%%%%%%%%%%%%%%%%%%%%%%%%%%%%%%%%%%%%%%%%%%%%%%%%%%%%%%%%%%%%%%%%%%%%%%%%%%%%%%%%%%%%%%%%%%%%%%%%%%%%%%%%%%%%%%%%%%%%%%%%%%%%%
\section{A variational approach}

In the previous sections, we studied the problem of minimizing~\eqref{equation.definition of E}.  As noted in the introduction, such minimization can be used to ``solve" an overdetermined system of differential equations~\eqref{equation.system of differential equations} in the least-squares sense of minimizing~\eqref{equation.least squares regularization}, even when the needed data $w(t)$ is only available at a single instant $t_0$.  There, we treated the unknown quantities $x(t_0)$ and $\dot{x}(t_0)$ as \textit{independent} unknowns, which reduces the problem of minimizing~\eqref{equation.least squares regularization} to~\eqref{equation.nonlinear least squares} in the special case where $P=M$.  However, this simplification comes at a cost: in light of Theorem~\ref{theorem.uniqueness}, we expect that at least $M+P=2M$ measurements must be made in order to uniquely determine $x(t_0)$.  When only $N<2M$ measurements are available, we are therefore led to consider alternative approaches to the minimization of~\eqref{equation.least squares regularization} in which $\dot{x}(t)$ is properly treated as a quantity that depends on $x(t)$.

In particular, in this section, we assume that the data $w(t)$ is given over an interval of time $[t_0,t_1]$, and we seek a parameterized curve $x(t)$ that minimizes the integral of~\eqref{equation.least squares regularization} over time, namely:
\begin{equation*}
\hat{E}(x)
:=\int_{t_0}^{t_1}\sum_{n=1}^{N}\abs{\ip{\dot{x}(t)}{\nabla \varphi_n(x(t))}-w_n(t)}^2\, \rmd t.
\end{equation*}
More formally, letting $F^*(x)$ denote the analysis operator of the vectors $\set{f_n(x)}_{n=1}^{N}$, $f_n(x):=\nabla\varphi_n(x)$ and letting $\Omega\subseteq\bbR^M$ denote the set of points $x$ at which these vectors form a frame for $\bbR^M$, we seek the minimizer of the functional:
\begin{equation}
\label{equation.functional}
\hat{E}:\rmC^1([t_0,t_1],\Omega)\rightarrow\bbR,
\qquad
\hat{E}(x):=\int_{t_0}^{t_1}\norm{F^*(x(t))\dot{x}(t)-w(t)}^2\,\rmd t.
\end{equation}
In the following result, we use techniques of variational calculus to find the Euler-Lagrange equation that any minimizer of~\eqref{equation.functional} necessarily satisfies. 
\begin{theorem}
\label{theorem.variation}
Let $\varphi_n\in\rmC^2(\Omega)$ for each $n=1,\dotsc,N$, and $w\in\rmC^1([t_0,t_1],\bbR)$.  Then any minimizer $x$ of \eqref{equation.functional} necessarily satisfies the Euler-Lagrange equation:
\begin{equation}
\label{equation.Euler-Lagrange}
F(x(t))\frac{\rmd}{\rmd t}\bigbracket{F^*(x(t))\dot{x}(t)-w(t)}=0,
\end{equation}
for all $t\in[t_0,t_1]$.
\end{theorem}

\begin{proof}
We write the functional~\eqref{equation.functional} as:
\begin{equation*}
\hat{E}(x)=\int_{t_0}^{t_1}\hat{e}(t,x(t),\dot{x}(t))\,\rmd t,
\end{equation*}
where the integrand function is:
\begin{equation*}
\hat{e}:[t_0,t_1]\times\Omega\times\bbR^M\rightarrow\bbR,\qquad \hat{e}(t,x,v)=\norm{F^*(x)v-w(t)}^2.
\end{equation*}
Note that since each entry of $F^*(x)$ is a first partial derivative of some $\varphi_n\in\rmC^2(\Omega)$, we necessarily have that $F^*(x)$ is continuously differentiable over $\Omega$.  This, combined with the assumption that $w\in\rmC^1([t_0,t_1],\bbR)$ and the fact that $\hat{e}$ is quadratic in the entries of $v$, implies that $\hat{e}$ is continuously differentiable over $[t_0,t_1]\times\Omega\times\bbR^M$.  Classical results from the calculus of variations then tell us that the functional $\hat{E}$ is Fr\'{e}chet-differentiable over $\rmC^1([t_0,t_1],\Omega)$, and that any minimizer of $\hat{E}$ necessarily satisfies the following system of $M$ Euler-Lagrange equations:
\begin{equation}
\label{equation.proof of Euler-Lagrange 1}
0=\frac{\partial\hat{e}}{\partial x_m}(t,x(t),\dot{x}(t))-\frac{\rmd}{\rmd t}\frac{\partial\hat{e}}{\partial v_m}(t,x(t),\dot{x}(t)),\quad\forall m=1,\dotsc,M.
\end{equation}
For any $(t,x,v)\in[t_0,t_1]\times\Omega\times\bbR^M$, the product rule gives the partial derivative of $\hat{e}$ with respect to the $m$th spatial variable to be:
\begin{align}
\nonumber
\frac{\partial\hat{e}}{\partial x_m}(t,x,v)
&=\frac{\partial}{\partial x_m}\norm{F^*(x)v-w(t)}^2\\
\nonumber
&=\frac{\partial}{\partial x_m}\bigip{F^*(x)v-w(t)}{F^*(x)v-w(t)}\\
\nonumber
&=2\biggip{F^*(x)v-w(t)}{\frac{\partial}{\partial x_m}\bigbracket{F^*(x)v-w(t)}}\\
\label{equation.proof of Euler-Lagrange 2}
&=2\biggip{F^*(x)v-w(t)}{\frac{\partial F^*}{\partial x_m}(x)v}.
\end{align}
Similarly, the partial derivative with respect to the $m$th momentum variable is:
\begin{align}
\nonumber
\frac{\partial\hat{e}}{\partial v_m}(t,x,v)
&=2\biggip{F^*(x)v-w(t)}{\frac{\partial}{\partial v_m}\bigbracket{F^*(x)v-w(t)}}\\
\label{equation.proof of Euler-Lagrange 3}
&=2\bigip{F^*(x)v-w(t)}{F^*(x)\delta_m}.
\end{align}
For a parameterized curve $x\in\rmC^1([t_0,t_1],\Omega)$ that minimizes $\hat{E}$, evaluating~\eqref{equation.proof of Euler-Lagrange 2} and~\eqref{equation.proof of Euler-Lagrange 3} at triples of the form $(t,x,v)=(t,x(t),\dot{x}(t))$ and then substituting the results into~\eqref{equation.proof of Euler-Lagrange 1} gives:
\begin{align}
\nonumber
0&=\biggip{F^*(x(t))\dot{x}(t)-w(t)}{\frac{\partial F^*}{\partial x_m}(x(t))\dot{x}(t)}\\
\label{equation.proof of Euler-Lagrange 4}
&\qquad-\frac{\rmd}{\rmd t}\bigip{F^*(x(t))\dot{x}(t)-w(t)}{F^*(x(t))\delta_m},\quad\forall m=1,\dotsc,M.
\end{align}
To simplify, note that the product rule gives:
\begin{align*}
\frac{\rmd}{\rmd t}\bigip{F^*(x(t))\dot{x}(t)-w(t)}{F^*(x(t))\delta_m}
&=\biggip{\frac{\rmd}{\rmd t}\bigbracket{F^*(x(t))\dot{x}(t)-w(t)}}{F^*(x(t))\delta_m}\\
&\quad+\biggip{F^*(x(t))\dot{x}(t)-w(t)}{\frac{\rmd}{\rmd t}F^*(x(t))\delta_m}.
\end{align*}
Substituting this expression into~\eqref{equation.proof of Euler-Lagrange 4} and collecting common terms gives:
\begin{align}
\nonumber
0&=\biggip{F^*(x(t))\dot{x}(t)-w(t)}{\frac{\partial F^*}{\partial x_m}(x(t))\dot{x}(t)-\frac{\rmd}{\rmd t}F^*(x(t))\delta_m}\\
\label{equation.proof of Euler-Lagrange 5}
&\qquad-\biggip{\frac{\rmd}{\rmd t}\bigbracket{F^*(x(t))\dot{x}(t)-w(t)}}{F^*(x(t))\delta_m},\quad\forall m=1,\dotsc,M.
\end{align}
At this point, it suffices for us to show that:
\begin{equation}
\label{equation.proof of Euler-Lagrange 6}
\frac{\rmd}{\rmd t}F^*(x(t))\delta_m=\frac{\partial F^*}{\partial x_m}(x(t))\dot{x}(t).
\end{equation}
Indeed, substituting~\eqref{equation.proof of Euler-Lagrange 6} into~\eqref{equation.proof of Euler-Lagrange 5} yields:
\begin{align*}
0
&=\biggip{\frac{\rmd}{\rmd t}\bigbracket{F^*(x(t))\dot{x}(t)-w(t)}}{F^*(x(t))\delta_m}\\
&=\biggip{F(x(t))\frac{\rmd}{\rmd t}\bigbracket{F^*(x(t))\dot{x}(t)-w(t)}}{\delta_m},\quad\forall m=1,\dotsc,M,
\end{align*}
which is equivalent to our claim \eqref{equation.Euler-Lagrange}.  To show that the vector equation \eqref{equation.proof of Euler-Lagrange 6} holds, note that the $n$th coordinate of the left-hand side is:
\begin{equation*}
\biggip{\frac{\rmd}{\rmd t}F^*(x(t))\delta_m}{\delta_n}
=\frac{\rmd}{\rmd t}\bigip{F^*(x(t))\delta_m}{\delta_n}
=\frac{\rmd}{\rmd t}\bigip{\delta_m}{f_n(x(t))}.
\end{equation*}
Now recall that the $n$th frame vector $f_n(x)$ is defined as the gradient of the $n$th scalar-valued function $\varphi_n(x)$.  As such, its $m$th entry is $\frac{\partial\varphi_n}{\partial x_m}(x)$, implying:
\begin{equation}
\label{equation.proof of Euler-Lagrange 7}
\biggip{\frac{\rmd}{\rmd t}F^*(x(t))\delta_m}{\delta_n}
=\frac{\rmd}{\rmd t}\frac{\partial\varphi_n}{\partial x_m}(x(t))
=\biggip{\dot{x}(t)}{\Bigparen{\nabla\frac{\partial\varphi_n}{\partial x_m}}(x(t))}.
\end{equation}
Now, since $\varphi_n\in\rmC^2(\Omega)$ by assumption, we know its second partial derivatives are symmetric, implying:
\begin{equation*}
\Bigparen{\nabla\frac{\partial\varphi_n}{\partial x_m}}(x)
=\Bigparen{\frac{\partial}{\partial x_m}\nabla\varphi_n}(x)
=\frac{\partial f_n}{\partial x_m}(x),\quad\forall x\in\Omega.
\end{equation*}
Using this fact in~\eqref{equation.proof of Euler-Lagrange 7} yields:
\begin{equation*}
\biggip{\frac{\rmd}{\rmd t}F^*(x(t))\delta_m}{\delta_n}
=\biggip{\dot{x}(t)}{\frac{\partial f_n}{\partial x_m}(x(t))}
=\biggip{\frac{\partial F^*}{\partial x_m}(x(t))\dot{x}(t)}{\delta_n},
\end{equation*}
which, since it holds for all $n=1,\dotsc,N$, implies our claim \eqref{equation.proof of Euler-Lagrange 6}.
\end{proof}

We conclude by noting that Theorem~\ref{theorem.variation} does not give an explicit algorithm for finding the minimizer of~\eqref{equation.functional}.  Rather, it provides a first-order, nonlinear differential equation that any minimizer must satisfy.  In practice, one way to make use of this result is to guess a large number of initial position and velocity combinations $(x(t_0),\dot{x}(t_0))$; for each combination, we can then use a numerical differential equation solver to extrapolate the target's trajectory $x:[t_0,t_1]\rightarrow\bbR^M$ according to~\eqref{equation.Euler-Lagrange}; we then choose the particular curve $x$ whose corresponding value $\hat{E}(x)$ is minimal.  We leave further developments of these ideas for future research.

\section*{Acknowledgments}
We thank Laura Suzuki, William Sturgis and the anonymous reviewer for their enlightening comments and suggestions.  This work was supported by NSF DMS 1042701, NSF CCF 1017278, AFOSR F1ATA01103J001, AFOSR F1ATA00183G003 and the A.B.~Krongard Fellowship.  The views expressed in this article are those of the authors and do not reflect the official policy or position of the United States Air Force, Department of Defense or the U.S.~Government.

\end{document}